\documentclass[12pt]{article}

\usepackage[bookmarks=true, bookmarksopen=true,%
bookmarksdepth=3,bookmarksopenlevel=2,%
colorlinks=true,%
linkcolor=blue,%
citecolor=blue,%
filecolor=blue,%
menucolor=blue,%
urlcolor=blue]{hyperref}

\usepackage[margin=1.3in,marginparwidth=0.8in, marginparsep=0.1in]{geometry}

\pdfoutput=1

\usepackage{amssymb,latexsym,amsmath,graphicx,tikz-cd}
\usepackage{bbm}
\usepackage{stmaryrd} 
\usepackage{amscd}
\usepackage{cite}
\usepackage{todonotes}
\usepackage{color}
\usepackage[all,cmtip]{xy}

\usepackage{amsfonts} 
\usepackage{mathpazo}

\usepackage{times}
\usepackage{mathrsfs}
\usepackage{mathtools}
\usepackage{parskip}

\newtheorem{lemma}{Lemma}[section]
\newtheorem{proposition}[lemma]{Proposition}
\newtheorem{corollary}[lemma]{Corollary}
\newtheorem{theorem}[lemma]{Theorem}

\newtheorem{definition}[lemma]{Definition}

\newtheorem{remark}{Remark}

\newenvironment{proof}{{\bf
		Proof.}}{$\blacksquare$ \vspace{2mm}}

\newcommand{\C}{\mathbb{C}}
\newcommand{\D}{\mathbb{D}}

\renewcommand{\H}{\operatorname{H}}

\newcommand{\N}{\mathbb{N}}

\newcommand{\Q}{\mathbb{Q}}
\newcommand{\R}{\mathbb{R}}

\newcommand{\Z}{\mathbb{Z}}

\newcommand{\cL}{\mathcal{L}}

\newcommand{\fd}{\frak{d}}

\newcommand{\fg}{\frak{g}}

\newcommand{\genkel}{\mathscr{F}}
\newcommand{\cotan}{T^*}
\newcommand{\flavor}{\mathbf{A}}
\newcommand{\flavorlie}{\fa}
\newcommand{\dualflavor}{\flavor^{\vee}}

\newcommand{\XX}{X}

\DeclareMathOperator{\Spec}{Spec}
\DeclareMathOperator{\Proj}{Proj}

\DeclareMathOperator{\cst}{cst}

\newcommand{\rk}{\operatorname{rk}}

\newcommand{\gvc}{\fg^{\vee}}

\newcommand{\gr}{\eta} 
\newcommand{\gc}{\lambda} 

\newcommand{\loc}{\operatorname{loc}}

\newcommand{\loops}{\mathscr{L}}
\newcommand{\tloops}{\widetilde{\loops}}

\newcommand{\fixed}{p}

\newcommand{\Rep}{\operatorname{Rep}}

\newcommand{\edges}{E}

\newcommand{\gaugeG}{\mathbf{G}}

\newcommand{\diag}{\mathbf{D}}

\newcommand{\eq}{\operatorname{eq}}

\newcommand{\per}{\mathscr{P}}

\newcommand{\ind}{\operatorname{ind}}

\newcommand{\coordset}{E}

\DeclareMathAlphabet\mathbfcal{OMS}{cmsy}{b}{n}

\newcommand{\bases}{\mathbb{B}}

\newcommand{\kstab}{\operatorname{Stab}}

\newcommand{\leaf}{\operatorname{Attr}}

\newcommand{\pic}{\operatorname{Pic}}
\newcommand{\finsymp}{\flavor}

\newcommand{\fa}{\frak{a}}

\newcommand{\ellcoh}{\operatorname{Ell}}

\newcommand{\bigtor}{T \times \C^\times_\hbar}

\newcommand{\elliptic}{\mathcal{E}}

\DeclareMathOperator{\oppo}{opp}

\author{Michael McBreen$^{[1]}$, Artan Sheshmani$^{[2,3,4]}$, Shing-Tung Yau$^{[5]}$}
\begin{document}
	
	\title{Elliptic stable envelopes and hypertoric loop spaces}
	\maketitle
	\smallskip
	
	\begin{center}
		
	\end{center}
	\begin{abstract}
		This paper describes a relation between the elliptic stable envelopes of a hypertoric variety $\XX$ and a distinguished $K$-theory class on the product of the loop hypertoric space $\tloops \XX$ and its symplectic dual $\per X^!$. This class intertwines the K-theoretic stable envelopes in a certain limit. Our results are suggestive of a possible categorification of elliptic stable envelopes.

		\smallskip
		

\end{abstract}
	\tableofcontents

	\section{Introduction} 

	Let $X$ be a smooth quasi-projective complex variety with an action of a torus $\mathbf{A}$. A polarisation is a class $T^{1/2} \in K_{\mathbf{A}}(X)$ satisfying
	\[ TX = T^{1/2}_X + (T^{1/2}_X)^{\vee}. \] 	
	Stable envelopes are certain correspondences between $X^{\mathbf{A}}$ and $X$ associated with a polarization and a choice of generic cocharacter of $\mathbf{A}$. They can be defined on the level of cohomology, K-theory or elliptic cohomology. They play a crucial role in modern geometric representation theory.
	
	In each case, one works equivariantly with respect to a larger group $\mathbf{T} \supset \mathbf{A}$, which typically does not admit a polarization. In each case, the existence and uniqueness of stable envelopes is ensured by certain conditions on the action of $\mathbf{A}$ on $X$.

Stable envelopes were first studied in the setting where $X$ is a conical symplectic resolution with an action of a Hamiltonian torus $\mathbf{A}$, contained in a larger (non-Hamiltonian) torus $\mathbf{T} = \mathbf{A} \times \C^{\times}_\hbar$. In particular, elliptic stable envelopes were first defined in \cite{aganagic2016elliptic} as classes in the elliptic cohomology of $X$ over the Tate curve, where $X$ is a Nakajima quiver variety or a hypertoric variety. The latter will be the focus of this paper; they play the same role in the theory of symplectic resolutions as toric varieties do in the theory of algebraic varieties.

In this setting, Aganagic and Okounkov \cite{okounkovstringmath2019talk, okounkovenumerativedualitytalk} have proposed an interpretation of elliptic stable envelopes in terms of a `duality interface' relating $X$ to the symplectically dual space $X^!$. This is an elliptic cohomology class $\frak{m}$ on $X \times X^!$ which gives rise to the elliptic stable envelopes on $X, X^!$ after restriction to torus fixed points on either side. They gave an explicit formula for this class when $X$ is hypertoric. Smirnov and Zhou \cite{Smzhou20} have developed the hypertoric duality interface in detail, and Rim\'anyi, Smirnov, Varchenko and Zhou \cite{rimanyi20193d} have described certain non-abelian examples.  
		
	In the spirit of the classical uniformization of theta functions over the Tate curve, we view the elliptic class $\frak{m}$ as an element of \begin{equation} \label{eq:ktheoryofXX} K_{\flavor \times \gaugeG^{\vee} \times \C^\times_\hbar}(X \times X^!)[[q]], \end{equation} i.e. as a $q$-series in the equivariant K-theory of $X \times X^!$. 
		
		Elliptic cohomology near the Tate curve is intimately related related to the loop space of $X$ - see for example \cite{ando2000power}, \cite{kitchloo2019quantization}. As discussed in \cite{okounkovstringmath2019talk}, the same should be true of the elliptic stable envelopes and the duality interface. The latter is expected to categorify to an equivalence between certain quasi-coherent sheaves on the loop spaces of $X$ and $X^!$ with `half-dimensional' support.
		
		Our aim in this paper is to reinterpret the duality interface on $X \times X^!$ as a $K$-theory class $\xi(\loops^{+})$ on $\tloops X \times \per X^!$, where $\tloops X$ is a hypertoric model of the loop space introduced in \cite{MasheYau20}, and $\per X^!$ is its symplectic dual. We explain how this class is an instance of a general construction in hypertoric symplectic duality. Our construction has the advantage of being both elementary and explicit. On the other hand, it is not clear whether it generalises beyond the hypertoric setting. 
		
		In more detail, recall from \cite{MasheYau20} that $\tloops X$ is defined as a limit of finite dimensional hypertoric varieties $\tloops_N X$ along closed embeddings, depending on a stability parameter $\eta$. 
		
		As explained in that paper, $\tloops X$ should be viewed as a first approximation to the universal cover of the space of loops into $X$, which captures some of its key geometric features. In particular, it carries an action of $H_2(X, \Z)$ corresponding to the action of the fundamental group of the loop space by deck transformations, and an action of $\C^\times_q = \C^\times$ corresponding to `loop rotation', whose fixed locus is identified with $X \times H_2(X, \Z)$. We define a certain completion $\tloops K(X)$ of the equivariant K-theory of $\tloops X$. The variable $q$ appears naturally as a character of the group $\C_q^\times$.		
		
		On the other hand, the space $\per X^!$ is a kind of multiplicative hypertoric space, originally studied in an unpublished note of Hausel and Proudfoot. It may be defined as a limit along open embeddings of smooth finite dimensional hypertorics $\per_N X^!$, depending on a stability parameter $\widetilde{\zeta}$. 
		
		In order to make contact with elliptic stable envelopes, we must consider a special value $\hat{\zeta}$ of the stability parameter for which the hypertorics $\per_N X^!_{\hat{\zeta}}$ are rather singular : they are quotients of $X^!_\zeta$ by certain torsion subgroups of its Hamiltonian torus. As a consequence, the limit space no longer exists in the category of schemes. We can nevertheless define a ring $\per K^{\circ}(X^!)$ which plays the role of ``equivariant K-theory" of the limit space.

		There is a geometrically defined map 
		\begin{equation} \label{eq:introphimap} \tloops K(X) \hat{\otimes} \per K(X^!) \to K_{\flavor \times \gaugeG^{\vee} \times \C^\times_\hbar}(X \times X^!)((q)) \end{equation} 
		 where the left-hand side denotes a certain completion of the tensor product. We find that the duality interface naturally lifts to a distinguished class $\xi(\loops^{+})$ in the left-hand ring.
		 
	To better understand  $\xi(\loops^{+})$, we observe that it is an instance of a much more general hypertoric construct. We define by a simple prescription a class $\xi \in K_{\flavor \times \gaugeG^{\vee} \times \C^\times_\hbar}(Y \times Y^!)$ associated to any pair of symplectically dual hypertorics $Y, Y^!$, together with a choice of polarisation. When $Y = \tloops X, Y^! = \per X^!$ and the polarisation is by holomorphic loops, we recover $\xi = \xi(\loops^{+})$. This is Theorem \ref{thm:maintheorem}. 
	
	We show that the class $\xi$ satisfies a number of properties analogous to the K-theoretic stable envelope. Proposition \ref{prop:intertwinertheorem} shows that when viewed as a correspondence from $Y$ to $Y^!$, it intertwines the K-theoretic stable envelopes of both spaces, once we let our equivariant parameters tend to infinity.
		
	The space $K_{\flavor \times \gaugeG^{\vee} \times \C^\times_\hbar}(Y \times Y^!)$ admits a tautological categorification, namely the derived category of equivariant coherent sheaves. There is a natural lift of $\xi$ to an object of this category, satisfying certain compatibilities with the various group actions. Our result is thus suggestive of a possible categorification of the elliptic stable envelope as a Fourier-Mukai kernel between the dual loop spaces of $X$ and $X^!$, as predicted in \cite{okounkovstringmath2019talk}.

	\subsection{Acknowledgements}
The first named author thanks Mina Aganagic, Alexander Braverman, Andrei Okounkov, Sarah Scherotzke, Nicol\`o Sibilla and Andrei Smirnov for helpful discussions. We also thank the anonymous referee for comments and corrections. Research of A. S. was partially supported by the US grants: NSF DMS-1607871, NSF DMS-1306313, Simons 38558, as well as Laboratory of Mirror Symmetry NRU HSE, RF Government grant, ag. No 14.641.31.0001. Furthermore, research of M. M. was carried at Harvard CMSA through Harvard CMSA / Aarhus Universitet collaboration agreement, and was supported by the supplemental grant of A. S. at Institut for Matematik, Aarhus Universitet. S.-T. Y. was partially supported by the US grants: NSF DMS-0804454, NSF PHY-1306313, and Simons 38558.
	
	\section{K-theoretic stable envelopes}
	
		The next few sections collect some generalities which we will have use for. We start by recalling the definition of a symplectic resolution, before narrowing our focus to hypertoric varieties in the main body of the paper.	
	
		\begin{definition} \label{def:sympres}
		Let  $\XX$ be a smooth complex variety equipped with an algebraic symplectic form $\Omega$ and an action of $\C^\times_\hbar := \C^\times$ scaling $\Omega$ by a nontrivial character $\hbar$. We call $\XX$ a conical symplectic resolution if 
		\begin{itemize}
			\item The natural map $\XX \to \Spec H^0(\XX, \mathscr{O}_\XX)$ is a projective resolution of singularities.
			\item The induced $\C^{\times}$-action on $\Spec H^0(\XX, \mathscr{O}_\XX)$ contracts it to a point.
		\end{itemize}
	\end{definition}
	
	We fix a maximal torus $\flavor$ of the group of (complex) hamiltonian automorphisms of $\XX$, which we assume has isolated fixed point locus $\XX^\flavor$.

	\subsection{Equivariant K-theory}
	
	Let $K_{\bigtor}(\XX)$ denote the equivariant K-theory ring of $\XX$. Let $K_{\bigtor}(pt)_{\loc}$ be the field of fractions of $K_{\bigtor}(pt)$.
	\begin{definition}
		Let $K_{\bigtor}(\XX)_{\loc} := K_{\bigtor}(\XX) \otimes_{K_{\bigtor}(pt)} K_{\bigtor}(pt)_{\loc}$.
	\end{definition}
 The equivariant Euler characteristic defines a map $\chi_{\eq} : K_{\bigtor}(\XX)\to K_{\bigtor}(\text{pt})_{\loc}$. We define a symmetric pairing on equivariant K-theory as follows.
		
	\begin{definition}
		Given $\gamma, \gamma' \in K_{\bigtor}(\XX)$, let $$\langle \gamma, \gamma' \rangle := \chi_{\eq}(\gamma\otimes \gamma').$$
	\end{definition}

In order to work with stable envelopes, we need a notion of degree as follows. Let $A$ be a torus. Any element $\genkel \in K_{A}(pt)$ can be expanded as a sum of characters \[ \genkel = \sum_{\mu \in X^{\bullet}(A)} a_\mu t^\mu. \] 
\begin{definition}
	We write $\deg_{A} \genkel$ for the convex hull in $\frak{a}^{\vee}_\R$ of the $A$-weights $\mu$ appearing with nonzero coefficient. 
\end{definition}
Degrees are partially ordered by containement of polytopes. 

We will also occasionally take limits of equivariant parameters, in the following sense. Any cocharacter $\sigma : \C^\times \to A$ determines a fan in $\mathfrak{a}_\R$ consisting of a single ray spanned by the derivative of $\sigma$. There is an open toric embedding $A \to A_\sigma$, where $A_\sigma$ is the toric variety associated to this fan. In coordinates, $\C[A_\sigma] \subset \C[A]$ is generated by $t^{\mu}$ satisfying $\langle \mu, \sigma \rangle \geq 0$. We have an isomorphism of schemes $A_\sigma \setminus A \cong A/\C^\times$.
\begin{definition} Let $\mathscr{Q}$ be an element of the fraction field of $H^0(A, \mathscr{O})$ with non-negative valuation along $A_\sigma \setminus A$. Write 
\[ \lim_{\sigma \to \infty} \mathscr{Q} \] for the corresponding element of the fraction field of $H^0(A_\sigma \setminus A, \mathscr{O}) \cong H^0(A/\C^\times, \mathscr{O})$. 
\end{definition}

	The following is elementary:
	\begin{lemma} \label{lem:strictboundvanishinglimit}
		If $\deg_A(\genkel)$ is strictly contained in $\deg_A(\mathscr{G})$, then the limit of $\genkel / \mathscr{G}$ along any cocharacter of $A$ vanishes.
	\end{lemma}
	This motivates the following definition.
\begin{definition}
We say $\genkel / \mathscr{G}$ is bounded if $\deg_A(\genkel)$ is contained in  $\deg_A(\mathscr{G})$, and strictly bounded if the containment is strict.
\end{definition}

\subsection{K-theoretic stable envelopes}
	We recall the definition of $K$-theoretic stable envelopes in a somewhat restricted generality, which will be sufficient for our purposes and simplifies the exposition. More details can be found in \cite[Section 9]{okounkov2015lectures}.
	
		We fix the following data:  
	\begin{enumerate}
		\item A cocharacter $\sigma$ of $\finsymp$, which is {\em generic} in the sense that $\XX^{\C^\times} = \XX^\flavor$.
		\item A polarization, i.e. a splitting
		\[ T\XX = T^{1/2} + \hbar^{-1} (T^{1/2})^{\vee} \] in $K_{\flavor \times \C^\times_\hbar}(\XX)$. 
		\item A {\em slope} $\cL \in \pic_{\finsymp}(\XX) \otimes_\Z \Q$, {\em generic} in the sense that the degree of $\cL$ on any rational curve joining two fixed points is non-integral.
	\end{enumerate}

 For $p \in \XX^\finsymp$, we can define the attracting cell
	\[ \leaf_\sigma(p) := \{ x \in \XX | \lim_{z \to 0} \sigma(z) \cdot x = p \}. \]
We define a partial order on $\XX^\finsymp$ by taking the closure of the relation $\{ q \leq p \text{ if } q \in \overline{\leaf_\sigma(p)} \}$. We define the `full attracting set' of $p$ to be 
	\[ \leaf_\sigma^{f}(p) := \cup_{q \leq p} \leaf_\sigma(q). \]
	It is a closed singular langrangian in $\XX$.

The K-theoretic stable envelope $\kstab_{\sigma, T^{1/2}, \cL}(p) \in K_{\finsymp \times \C^\times_\hbar}(\XX)$ is a class satisfying the following conditions :
	\begin{enumerate}
		\item It is supported on $\leaf^f_\sigma(p).$ 
		\item Its restriction to $p$ equals the restriction of $\mathscr{O}_{\leaf_p} \otimes \mathscr{L}$ where 
		$$\mathscr{L} =  (-1)^{\rk T^{1/2}_{>0}} \left( \frac{ \det T_{< 0} } { \det T^{1/2} } \right)^{1/2}$$
		Here $T_{<0}$ is the repelling part of the tangent space at $p$, i.e. the complement to the tangent space of $\leaf_\sigma(p)$.
		\item Let $q \in \XX^\flavor$. Then we have
		\[ \deg_{\flavor} \kstab(p)|_q \otimes \cL_p \subset \deg_{\flavor} \kstab(q)|_q \otimes \cL_q. \]
	\end{enumerate}
	Stable envelopes exist, and are uniquely specified by the above conditions, for a wide class of symplectic resolutions including all hypertoric varieties. See \cite[Section 9]{okounkov2015lectures} for an introduction and \cite{okounkov2020inductive} for a much more general construction.  
\begin{definition}
 Let $A$ be a torus. Let $\overset{\bullet} \bigwedge : (K_{A}(\text{pt}), +) \to  (K_{A}(\text{pt})_{\text{loc}}, \otimes)$ be the unique map extending  $V\to \sum_{i}(-1)^{i}\bigwedge^{i}V$. It may be written in coordinates as $$\sum_{\mu\in X^{\bullet}(A)}c_{\mu}t^{\mu}\to \prod_{\mu\neq 0}(1-t^{\mu})^{c_{\mu}}.$$
 \end{definition}
Condition 2 tells us that, after specializing $K_{\flavor \times \C^\times_\hbar}(pt) \to K_{\flavor}(pt)$, we have an equality 
 \[ \kstab(p)|_p = \bigwedge^{\bullet} (T^{1/2}|_p)^{\vee}. \]

The following proposition tell us that stable envelopes for `opposite' choices of data form dual bases of K-theory.
\begin{proposition}	\label{prop:stablesaredual}
	Fix data $\sigma, T^{1/2}, \cL$ as above. Let $T^{1/2}_{\oppo} := T_X - T^{1/2}$. 
	\begin{enumerate}
		\item The classes $\operatorname{Stab}_{\sigma, T^{1/2}, \cL}(p)$ for $p\in X^{\flavor}$ form a basis of $K_{\bigtor}(X)_{\text{loc}}$ over $K_{\bigtor}(\text{pt})_{\text{loc}}$.
		\item $\left \langle \operatorname{Stab}_{\sigma, T^{1/2}, \cL}(p), \operatorname{Stab}_{-\sigma, T^{1/2}_{\oppo}, \cL^{-1}}(q)\right \rangle=\delta_{pq}$.
		\end{enumerate}
	\end{proposition}

		\section{Hypertoric varieties} \label{sec:hypertoricvarieties}
	In this section we define our main geometric actors: the hypertoric varieties introduced in \cite{bielawski2000geometry}. For a survey of these spaces, see \cite{proudfoot2006survey}. 
	
	Fix the following data:
	\begin{enumerate}
		\item A finite set $\coordset$.
		\item A short exact sequence of complex tori 
		\begin{equation} \label{eq:basicsequence} 1 \to \gaugeG \to \diag \to \flavor \to 1, \end{equation} with an isomorphism $\diag = (\C^\times)^\coordset$.
		\item A character $\eta$ of $\gaugeG$.
	\end{enumerate}
	To these choices we will associate a hypertoric variety. Let $\fg, \fd, \flavorlie$ be the complex lie algebras of $\gaugeG,\diag,\flavor$. We require that $\fd_{\Z} \to \flavorlie_{\Z}$ be totally unimodular, i.e. the determinant of any square submatrix (for a given choice of integer basis) is one of $-1,0,1$. This will ensure that when smooth, our hypertoric variety is a genuine variety and not an orbifold. We also assume that no cocharacter of $\gaugeG$ fixes all but one of the coordinates of $\C^\coordset$.
	
	Let $V := \Spec \C[x_e | e \in \coordset]$. The torus $\diag$ acts by hamiltonian transformations on $\cotan V = \Spec \C[x_e, y_e | e \in \coordset]$, equipped with the standard symplectic form $\Omega := \sum_{e \in \coordset} dx_e \wedge dy_e$.  A moment map $\mu_\diag: \cotan V\to \fd^{\vee}$ is given by
	\[
	\mu_\diag(x,y) = (x_e y_e).
	\]
	We have the exact sequence
	\begin{equation} \label{basicsequence}
	0\to\fg\overset{\partial}{\to}\fd \overset{\beta}{\to} \flavorlie\to 0
	\end{equation}
	and its dual
	\begin{equation} \label{dualsequence}
	0\to \flavorlie^{\vee}  \overset{\beta^{\vee}}\to \fd^{\vee}\overset{\partial^{\vee}}{\to} \fg^{\vee}\to 0.
	\end{equation}
	
	The pullback $\mu_\gaugeG=\partial^{\vee}\circ\mu_\diag$ defines a moment map for the $\gaugeG$ action on $\cotan V$. Fix a character $(\gr, \gc) \in \fg^{\vee}_\Z \oplus \gvc$. 		Given a $\gaugeG$-variety $U$, write $U \sslash_{\gr} \gaugeG$ for the GIT quotient $\Proj \bigoplus_{m \in \N} \{ f \in \mathscr{O}(U) : g^* f = \gr(g)^{-m} f. \}$. Let

	\begin{definition}
		\begin{align}
		\label{def:hypertoricreduction} X_{\gr, \gc} & := \mu_{\gaugeG}^{-1}(\gc) \sslash_{\gr} \gaugeG. 
		\end{align}
	\end{definition}
	We will often assume that $\eta$ is suitably generic, in which case $\XX_{\gr, \gc}$ is smooth; this holds away from a finite set of hyperplanes. We write $X_{\gr} := X_{\gr, 0}$, which we sometimes abbreviate further to $\XX$. 
	
	
	The Kirwan map defines morphisms
	$\H^2_\flavor(\XX, \Z) \to \fd^{\vee}_\Z$, $\H^2(\XX, \Z) \to \frak{g}^{\vee}_\Z$ and $\H_2(\XX , \Z) \to \fg_\Z$. If the map $\beta$ does not annihilate the natural basis of $\mathfrak{d} = \Z^n$ and $\eta$ is suitably generic, these maps are isomorphisms. 
	

	The variety $\XX$ inherits an algebraic symplectic structure from its construction via symplectic reduction. The induced $\flavor$ action on $\XX$ is Hamiltonian. There is a further action of $\C^{\times}_{\hbar}$ dilating the fibers of $\cotan V$, which scales the symplectic form by $\hbar$. This preserves $\mu_\gaugeG^{-1}(0)$, and descends to an action of $\C^{\times}_{\hbar}$ on $\XX$ commuting with the action of $\flavor$.

	\subsection{Bases and torus fixed points} 
	The torus fixed points $\XX^{\flavor}_\eta$ are indexed by {\em bases}. These are the subsets $b \subset \edges$ such that the restriction of $\Z^E \to \flavorlie_\Z$ to $\Z^b$ is an isomorphism. By construction, the set of bases $\mathbb{B}$ does not depend on $\eta$.  
	
	\begin{lemma}
		There is a bijection $\bases \to \XX_\eta^\flavor$ taking $b$ to $$\fixed_b := \left( \cotan  \C^{\edges \setminus b} \cap \mu_\gaugeG^{-1}(0) \right) \sslash_\eta \gaugeG.$$
	\end{lemma} 
We can schematically write $p = \bigcap_{e \in b} \{ x_e = y_e = 0 \}$.

	The isomorphism $\Z^b \to \flavorlie_\Z$ determines a basis of the right-hand lattice. Let $\{ \alpha^b_e \} \subset \flavorlie_\Z^{\vee}$ be the dual basis. We will sometimes write $\alpha^p_e$ if we wish to emphasise the fixed point rather than the base.
\begin{lemma} \label{lem:charofnormal}  Let $e \in b$. The normal bundle to $\{x_e = 0\}$ at $p$ has $\flavor$-character $\alpha^p_e$. The normal to $\{y_e = 0\}$ has $\flavor$-character $-\alpha^p_e$. \end{lemma}
	\begin{corollary} \label{def:zetapositive}
	Let $e \in b$. Then the normal to $\{ x_e = 0 \}$ at $p$ is attracting for the cocharacter $\zeta$ if $\langle \alpha^p_e, \zeta \rangle > 0$ and repelling if $\langle \alpha^p_e, \zeta \rangle < 0$. 
	\end{corollary}
	We now turn our attention to $e \notin b$, and characterise which of the divisors $\{ x_e = 0 \}$ or $\{ y_e = 0 \}$ contains $p$. The map $\fg_\Z \to \Z^{b^c}$ is an isomorphism. Dualizing gives a map $\Z^{b^c} \to \fg^{\vee}$, and thus a basis of $\fg^{\vee}_\Z$. We let $\beta^p_e$ be the dual basis of $\fg_\Z$. 

		\begin{lemma} \label{lem:etapositive}
		There is a unique coordinate lagrangian $L_\eta \subset \cotan  \C^{E \setminus b}$ containing an $\eta$-semistable point, cut out by $x_e = 0$ for $\langle \beta^p_e, \eta \rangle < 0$ and $y_e = 0$ for $\langle \beta^p_e, \eta \rangle > 0$ . We have  $p = L_\eta \sslash_\eta \gaugeG$.
	\end{lemma}
\begin{corollary}
Let $e \notin b$. Then $p \in \{ x_e = 0 \}$ if $\langle \beta^p_e, \eta \rangle < 0$ and $p \in \{ y_e = 0 \}$ if $\langle \beta^p_e, \eta \rangle > 0$
\end{corollary}

Fix a generic cocharacter $\zeta \in \flavorlie_\Z$. 
	\begin{definition} \label{def:vermalag}
		Let $\leaf_\zeta^n(p) \subset X$ be the singular lagrangian defined by intersecting $\{ y_e = 0 \}$ for $\langle \zeta, \alpha^p_e \rangle > 0$ with $\{ x_e = 0 \}$ for $\langle \zeta, \alpha^p_e \rangle < 0$.
		\end{definition}
	It is a union of components of $\leaf_\zeta^f(p)$, and is precisely the support of the $K$-theoretic stable envelope of $p$, although we will not use this fact below.
	
	We have the following useful characterisation of the fixed points which lie in this set.
		\begin{lemma} \label{lem:coordinatesofattn}
 Let $b_p, b_q$ be the bases associated to $p,q \in X^\flavor$.
$q \in \leaf_{\zeta}^n(p)$ if and only if $\langle \alpha^p_e, \zeta \rangle \langle \beta^q_e, \eta \rangle > 0$ for all $e \in b_q \cap b_p^c$. 
\end{lemma} 
	\subsection{Symplectic duality for polarized hyperplane arrangements, or Gale duality}  \label{sec:galedualhypertoric}
Symplectic duality as defined in \cite{braden2014quantizations} may be thought of as a relation between two symplectic resolutions (or more generally, symplectic singularities). We refer to that paper for the general concept: here we will content ourselves with a review of the construction of the symplectic dual of a hypertoric variety $\XX_\eta$, in order to fix notation. 

Consider a sequence of tori as in \ref{eq:basicsequence}, together with a character $\eta$ of $\gaugeG$ and a cocharacter $\zeta$ of $\flavor$. We define the Gale dual data to be
	\begin{enumerate}
		\item The set $\edges$.
		\item The dual sequence of tori 
		\begin{equation} \label{eq:dualbasicsequence} 1 \to \dualflavor  \to \diag^{\vee} \to \gaugeG^{\vee} \to 1 \end{equation} 
		with the induced isomorphism $\diag^{\vee} \cong (\C^\times)^{\edges}$.
		\item The character $-\zeta$ of $\dualflavor $.
		\item The cocharacter $-\eta$ of $\gaugeG^{\vee}$.
	\end{enumerate}
	The torus $\diag^{\vee}$ acts on $T^* V^{\vee}$. We define $X_{-\zeta}^!$ as the symplectic reduction of $\cotan  V^{\vee}$ by the induced action of $\dualflavor $ with GIT parameter $-\zeta$. We will write $\check{x}_e, \check{y}_e$ for the natural coordinates on $\cotan V^{\vee}$. 
	
	In general, we will use the shriek superscript to indicate that we are working with $X_{-\zeta}^!$ rather than $X_{-\zeta}$. In particular, we set $\edges^! := \edges$. There is a natural bijection of the bases $\bases \cong \bases^!$ given by taking $b \subset \edges$ to its complement $b^c \subset \edges$.
	
	\begin{definition}
		Given a fixed point $p \in \XX_{-\zeta}^\flavor$ indexed by $b \subset \edges$, we write $p^! \in (\XX_{-\zeta}^{!})^{\gaugeG^{\vee}}$ for the fixed point indexed by $b^c$.
	\end{definition} 
	
	The following is a direct consequence of the definitions.
	\begin{lemma} \label{lem:alphabeta}
	Let $e \in b$. Then $\alpha^p_e = \beta^{p^!}_e$.
	\end{lemma}

			\section{Cohomology and K-theory of hypertoric varieties}
			
			\subsection{The Kirwan map}
	\begin{definition} 
		Let $$\kappa : \Rep \diag \times \C^\times_\hbar = K_{\diag \times \C^\times_\hbar}(T^*V) \to K_{\flavor \times \C^\times_\hbar}(\mu_\gaugeG^{-1}(0) \sslash \gaugeG)$$ be the composition of the restriction to $\mu_\gaugeG^{-1}(0)$ with the Kirwan map, which takes a representation $R$ of $\diag \times \C^{\times}_\hbar$ to the class of the associated bundle $R \times^\gaugeG \mu_\gaugeG^{-1}(0)^{\gaugeG-ss}$.
	\end{definition} 
		
	\begin{definition}
		Given $e \in \coordset$, let $\chi_e$ be the $\diag \times \C^\times_\hbar$-character of $x_e \in \mathscr{O}(T^*V)$, and let \[ u_e := \kappa(\chi_e) \in K_{\flavor \times \C^\times}(X). \] 
	\end{definition}
	Thus $u_e$ represents an equivariant line bundle on $X$. The dual Darboux coordinate $y_e$ has character $\hbar^{-1} \chi_e^{-1}$, defining the bundle $\hbar^{-1} u_e^{-1}$.
	
	Let $\check{\chi}_e$ be the character of $\check{x}_e$ under $\diag^{\vee}$. We have the analogous definition:
	\begin{definition}
		$$\check{u}_e := \kappa(\check{\chi}_e) \in K_{\gaugeG^{\vee} \times \C^\times_\hbar}(X^!).$$ 
	\end{definition}

\begin{definition}
Given any coordinate Lagrangian subspace $L \subset \cotan  V$, we define a corresponding polarisation of $X$ by viewing $L$ as a representation of $\diag \times \C^\times$ and taking its image under the Kirwan map.	
\end{definition}
	
	\subsection{Restriction to a fixed point}
	
We recall some known facts about the classes $u_e, \check{u}_e$. The following is essentially a restatement of Lemma \ref{lem:charofnormal}.
	\begin{lemma}
		Let $p \in \XX^\flavor$ be indexed by the base $b$. Let $e \in b$. Then the image of $u_e|_p$ under the map $H^{\bullet}_{\flavor \times \C^\times_\hbar}(p, \C) \to H^{\bullet}_{\flavor}(p, \C)$ equals $\alpha^p_e$.
		\end{lemma}
\begin{lemma} \label{lem:resgiveshbar}
	Keep the notations of the previous lemma, but suppose $e \notin b$. We have $u_e|_p = \hbar$ if $\langle \beta_e, \eta \rangle > 0$, and $u_e|_p = 1$ if $\langle \beta_e, \eta \rangle < 0$.
\end{lemma}
We introduce the notation $\epsilon^p_e \in \Z$ for the function such that $u_e|_p = \alpha^p_e \hbar^{\epsilon^p_e}$ for $e \in b_p$ and $u_e|_p = \hbar^{\epsilon^p_e}$ for $e \notin b_p$. For $e \notin b$ we have $\epsilon^p_e = 0$ if and only if $\langle \beta^p_e, \eta \rangle < 0$.

\section{The class $\xi$}\label{sec:xi}
Consider the antidiagonal embedding $\C^{\times}_\hbar \to (\C^{\times}_\hbar)^2, z \to (z,z^{-1})$. We use this embedding to define the equivariant K-group $K_{\flavor \times \gaugeG^{\vee}\times \C^\times_{\hslash}}(X \times X^{!})$. Our main object of interest is the following class in this K-group.
\begin{definition}
$$ \xi := \prod_{e \in E} (1 - u_e \check{u}_e).$$

\end{definition}

We may think of $\xi$ as associated to the polarisation $V$ of $T^*V$, as follows. Recall that our construction of dual hypertorics in Sections \ref{sec:hypertoricvarieties} and \ref{sec:galedualhypertoric} starts from the tori $\diag\times \C^\times_\hbar, \diag^{\vee} \times \C^\times_\hbar$ acting on the spaces $\cotan V, \cotan V^{\vee}$. As torus representations, we have decompositions
\begin{align*}
V=\bigoplus_{e\in E} \chi_{e} \text{  and  } V^{\vee}=\bigoplus_{e\in E}\chi^{\vee}_{e}.
\end{align*}

\begin{definition}
	Let $$\tilde{\xi}:= \bigwedge^{\bullet} \left( \sum_{e}\chi_{e}\chi^{\vee}_{e} \right)$$
	viewed as an element of $K_{\diag\times \C^\times_\hbar \times \diag^{\vee} \times \C^\times_\hbar}(pt)$. 
\end{definition}	

Now we fix dual hypertorics $$X:=\cotan V\sslash_\eta \gaugeG, \,\,\,\,\,\,\,\, X^{!}:=\cotan V^{\vee}\sslash_\zeta \dualflavor $$ as in Sections \ref{sec:hypertoricvarieties} and \ref{sec:galedualhypertoric}. We also fix the auxiliary data which specifies stable envelopes on $X, X^!$. Thus, we fix suitably generic choices of $$\mathscr{L}^{X}\in \text{Pic}_{\flavor}(X)\otimes_{\mathbb{Z}}\mathbb{Q}, \,\,\,\,\,\, \mathscr{L}^{X^{!}}\in \text{Pic}_{\gaugeG^{\vee}}(X^{!})\otimes_{\mathbb{Z}}\mathbb{Q}.$$
Furthermore, we pick the usual polarizations $T^{1/2}_X$ (resp $T^{1/2}_{X^!}$) of $X_{\eta}$ (resp $X^{!}_{\zeta}$) induced by the image of $V$ (resp $V^{\vee}$) under the Kirwan map.

We have a Kirwan map
$ K_{\diag\times \C^\times_\hbar \times \diag^{\vee} \times \C^\times_\hbar}(pt) \to K_{\flavor \times \C^{\times}_\hbar \times \gaugeG^{\vee} \times \C^\times_\hbar}(X \times X^!).$
We further restrict along the antidiagonal embedding $\C^{\times}_\hbar \to (\C^{\times}_\hbar)^2, z \to (z,z^{-1})$ to obtain a map 
\begin{equation} \label{eq:antidiagkirwan}  K_{\diag\times \C^\times_\hbar \times \diag^{\vee} \times \C^\times_\hbar}(pt) \to K_{\flavor \times \gaugeG^{\vee} \times \C^\times_\hbar}(X \times X^!). \end{equation}

Thus  $\xi$ is the image of $\tilde{\xi}$ under the map \ref{eq:antidiagkirwan}.

\begin{definition} \label{rem:xi} 
Given any coordinate subspace $V' \subset \cotan V$, we define a class $\xi(V')$ as the image under the map $\ref{eq:antidiagkirwan}$ of the class
$$ \bigwedge^{\bullet} \left( \sum \chi_i \chi_i^{\vee} \right),$$ where $V' = \oplus \chi_i$ is the isotypic decomposition of $V'$ into $\diag$-characters.    \end{definition}

The class $\xi(V')$ satisfies a number of interesting properties analogous to the defining properties of the K-theoretic stable envelope. The rest of this section explores a few of these properties, which will not however be needed in the remainder of this paper. For simplicity, we focus our attention on $\xi = \xi(V)$ below.

\begin{lemma}
	Let $p,q \in X^\flavor$ such that $q \notin \leaf_\zeta^n(p)$ , or equivalently  $\langle \alpha^q_e, \zeta \rangle \langle \beta^p_e, \eta \rangle < 0$ for some $e \in b_q \cap b_p^c$.  Then $$u_e|_{p}\check{u}_e|_{q^!} = 1.$$
\end{lemma}
\begin{proof}
	This follows from Lemma \ref{lem:coordinatesofattn}, Lemma \ref{lem:alphabeta} and Lemma \ref{lem:resgiveshbar}.
	\end{proof}  
\begin{corollary} \label{cor:restrictionvanishes}
The restriction $\xi_{p \times q^!}$ vanishes unless $p \in \leaf_\zeta^n(q)$, or equivalently  $\langle \alpha^q_e, \zeta \rangle \langle \beta^p_e, \eta \rangle > 0$ for all $e \in b_q \cap b_p^c$.
\end{corollary}

\begin{lemma} \label{lem:xibound}
	\[\deg \xi_{p \times q^!}  \leq \deg \bigwedge^{\bullet} T^{1/2}_p X \otimes \bigwedge^{\bullet} T^{1/2}_{q^!} X^!. \]
	for $p \neq q$, and
	\[ \xi_{p \times p^!} = \bigwedge^{\bullet} T^{1/2}_p X \otimes \bigwedge^{\bullet} T^{1/2}_{p^!} X^!. \]
	Here all classes are taken equivariant with respect to the Hamiltonian subtorus $\flavor \times \gaugeG^{\vee} \subset \flavor \times \gaugeG^{\vee} \times \C^\times_\hbar$.
\end{lemma}
\begin{proof}
	We have \begin{equation} \label{eq:xirestricted} \deg \xi|_{p \times q^!} = \deg \prod_{e \in \edges} (1 - u_e|_p \check{u}_e|_{q^!}) \end{equation}
	Let $b_p, b_q \subset \edges$ be the bases associated to $p, q$ respectively, so that $b_q^c$ is the base associated to $q^!$.

	The characters $u_e|_p$ for $e \in b_p$ are precisely the summands of $T^{1/2}_p X$ with nonzero $\flavor$-weight, and likewise for $T^{1/2}_{q^!} X^!$. This proves the first inequality. When $p=q$, each factor contains a single nontrivial character of either $T^{1/2}_p X$ or $T^{1/2}_{q^!} X^!$, thus proving the second equality.    
	
\end{proof}

The following Proposition is the main result of this section. It shows that $\xi$ intertwines the stable envelopes of $X,X^!$ in a certain limit. Consider the cocharacter $\zeta \times \eta^{-1} : \C^\times \to \flavor \times \gaugeG^{\vee}$. It defines a map $$K_{\flavor \times \gaugeG^{\vee} \times \C^\times_\hbar}(X \times X^!) \to K_{\C^\times \times \C^\times_\hbar}(X \times X^!).$$ We may view elements of the right-hand space as functions of the tautological character $t \in K_{\C^\times}(pt)$.
\begin{proposition} \label{prop:intertwinertheorem}
The image in $ K_{\C^\times \times \C^\times_\hbar}(X \times X^!)$ of 
\begin{align*}
	 (\mathscr{L}^{X}_{p}\otimes \mathscr{L}^{X^{!}}_{q^!}) \otimes \left \langle (\mathscr{L}^{X})^{-1} \otimes \xi\otimes (\mathscr{L}^{X^{!}})^{-1}, \operatorname{Stab}_{\zeta, T^{1/2}_{X, \oppo}, \mathscr{L}^{X}}(p)\otimes \operatorname{Stab}_{\eta, T^{1/2}_{X^!, \oppo}, \mathscr{L}^{X^{!}}}(q^{!}) \right \rangle
	\end{align*}
has a limit as $t \to \infty$. This limit equals $\left(\frac{\hbar}{1 - \hbar}\right)^{\rk \ind_p} \left(\frac{\hbar^{-1}}{1 - \hbar^{-1}} \right)^{\rk \ind_{p^!}}$ if $p=q$ and equals $0$ otherwise. \end{proposition}
	Here $\ind_p = T^{1/2}_{p, >0}$ is the index bundle at $p$, and $\ind_{p^!} = T^{1/2}_{p^!, >0}$ is the index bundle at $p^!$. Note that $\rk \ind_p$ is the number of $e$ for which $\langle \alpha^{p}_e, \zeta \rangle  > 0$.

\begin{remark}
By Proposition \ref{prop:stablesaredual}, we may reinterpret the Proposition as follows. Viewed as a correspondence, $(\mathscr{L}^{X})^{-1} \otimes \xi\otimes (\mathscr{L}^{X^{!}})^{-1}$ defines a map from $K_{\flavor \times \C^\times_\hbar})(X)_{\loc}$ to $K_{\gaugeG^{\vee} \times \C^\times_\hbar}(X^!)_{\loc}$. In the basis defined by $K$-theoretic stable envelopes, this map is represented by a diagonal matrix in the limit $t \to \infty$.
\end{remark}

\begin{proof} 	Write 
	\begin{align*}
	A&:= (\mathscr{L}^{X})^{-1}\otimes \xi\otimes (\mathscr{L}^{X^{!}})^{-1}\notag\\
	B^{p,q^!}&:= (\mathscr{L}^{X}_{p}\otimes \mathscr{L}^{X^{!}}_{q^!}) \otimes \operatorname{Stab}_{\zeta, T^{1/2}_{X, opp}, \mathscr{L}^{X}}(p)\otimes \operatorname{Stab}_{\eta, T^{1/2}_{X^!, opp}, \mathscr{L}^{X^{!}}}(q^{!}).
	\end{align*}
	
	We use the  localization theorem for equivariant K-theory, which expresses our pairing as a sum over fixed points:
 \begin{equation} \label{eq:localization} \langle A, B^{p \times q^!}\rangle =\sum
	_{x\in X^{\flavor}, y^{!}\in (X^{!})^{\gaugeG^{\vee}}}\frac{A_{x\times y^{!}}\otimes B^{p \times q^!}_{x\times y^{!}}}{\bigwedge^{\bullet}(T_{x}X)^{\vee}\otimes \bigwedge^{\bullet}(T_{y^{!}}X^{!})^{\vee}}.\end{equation}
	We consider the right-hand side summand by summand. By Corollary \ref{cor:restrictionvanishes}, we may assume $x \in \leaf^n_\zeta(y)$.
	
	We have a restriction map $K_{\flavor \times \gaugeG^{\vee}}(X \times X^!) \to K_{\C^\times}(X \times X^!)$ induced by our choice of cocharacter. Below we write $\deg$ for the degree with respect to $\C^\times$. 
		By definition of the stable envelope, we have 
	$$\operatorname{deg} B^{p \times q^!}_{x\times y^{!}}\leq \operatorname{deg} \bigwedge^{\bullet}(T^{1/2}_{X,x})^{\vee}\otimes \bigwedge^{\bullet}(T^{1/2}_{X^{!},y^{!}})^{\vee}\otimes \mathscr{L}^X_{x} \otimes \left( \mathscr{L}^{X^!}_{y^!}  \right)^{-1} .$$
	with a strict inequality when $x \times y^! \neq p \times q^!$. On the other hand, by Lemma \ref{lem:xibound} we have  $\operatorname{deg}\xi\mid_{x\times y^{!}} \leq \operatorname{deg} \bigwedge^{\bullet}T^{1/2}_{X,x}\otimes \bigwedge^{\bullet}T^{1/2}_{X^{!},y^{!}}$. 
			
	Combining, we find that every summand on the right hand side of Equation \ref{eq:localization} is bounded, and those summands with  $x \times y^! \neq p \times q^!$ are strictly bounded. 
	
	Upon taking the limit $t \to \infty$, the summands with $x \times y^! \neq p \times q^!$ tend to zero by Lemma \ref{lem:strictboundvanishinglimit}. 
	
	To finish the proof, we must compute the limits of those summands with $x \times y^! = p \times q^!$. This is done in Lemma \ref{lem:pneqq} below.
\end{proof}

\begin{lemma} \label{lem:pneqq}
Let $F_{p \times q^!}(t, \hbar)$ be the image of $\xi_{p \times q^!} \cdot \left( \bigwedge^{\bullet} T^{1/2}_p X \otimes \bigwedge^{\bullet} T^{1/2}_{q^!} X^! \right)^{-1}$ in $K_{\C^\times \times \C^\times_\hbar}(pt)$. Suppose $p \neq q$. Then $\lim_{t \to \infty} F_{p \times q^!} = 0$. On the other hand,
$$\lim_{t \to \infty} F_{p \times p^!} = \left(\frac{\hbar}{1 - \hbar}\right)^{\rk \ind_p} \left(\frac{\hbar^{-1}}{1 - \hbar^{-1}} \right)^{\rk \ind_{p^!}}.$$ 
\end{lemma}
\begin{proof}
 $F_{p \times q^!}$ is given by a product of factors of the following form. Below, vanishing factors in the denominator of the form $(1-\hbar^{\epsilon^p_e})$ with $\epsilon^p_e=0$ are understood to be ommited.
\begin{equation} \label{eq:firstfactor} \frac{(1 - t^{\langle \alpha^{p}_e, \zeta \rangle } \hbar^{\epsilon^p_e} t^{\langle \alpha^{p^!}_e, -\eta \rangle } \hbar^{-\epsilon^{q^!}_e})}{(1 - t^{\langle \alpha^{p}_e, \zeta \rangle }\hbar^{\epsilon^p_e})(1-t^{\langle \alpha^{p^!}_e, -\eta \rangle }\hbar^{-\epsilon^{q^!}_e})} \end{equation}
for $e \in b_p \cap b_q^c$,
\begin{equation} \label{eq:secondfactor} \frac{(1 - t^{\langle \alpha^{p}_e, \zeta \rangle } \hbar^{\epsilon^p_e} \hbar^{-\epsilon^{q^!}_e})}{(1 - t^{\langle \alpha^{p}_e, \zeta \rangle }\hbar^{\epsilon^p_e})(1- \hbar^{-\epsilon^{q^!}_e})} \end{equation}
for $e \in b_p \cap b_q$,
\begin{equation} \label{eq:thirdfactor} \frac{(1 - \hbar^{\epsilon^p_e} t^{\langle \alpha^{p^!}_e, -\eta \rangle } \hbar^{-\epsilon^{q^!}_e})}{(1 - \hbar^{\epsilon^p_e})(1- t^{\langle \alpha^{p^!}_e, -\eta \rangle }\hbar^{-\epsilon^{q^!}_e})} \end{equation}
for $e \in b^c_p \cap b^c_q$ and
\begin{equation} \label{eq:lastfactor} \frac{(1 - \hbar^{\epsilon^p_e} \hbar^{-\epsilon^{q^!}_e})}{(1 - \hbar^{\epsilon^p_e})(1- \hbar^{-\epsilon^{q^!}_e})} \end{equation}
for $e \in b^c_p \cap b_q$. 
 We consider the each terms in the limit $t \to \infty$. By Corollary \ref{cor:restrictionvanishes}, we may assume $\langle \alpha^{p}_e, \zeta \rangle  \langle \alpha^{p^!}_e, -\eta \rangle  < 0$. Recall also that for $e \notin b$, $\epsilon^p_e \neq 0$ exactly when $\langle \alpha^{p^!}_e, -\eta \rangle  > 0$, and $\epsilon^{q^!}_e \neq 0$ precisely when $\langle \alpha^{p}_e, \zeta \rangle > 0$. Thus, in factors of type \ref{eq:lastfactor}, one of $\epsilon^p_e, \epsilon^{q^!}_e$ must vanish, and the factor equals one. 
 
On the other hand, the factors of type \ref{eq:firstfactor} vanish in the limit since  $\langle \alpha^{p}_e, \zeta \rangle  \langle \alpha^{p^!}_e, -\eta \rangle  < 0$. The factors of type \ref{eq:secondfactor} limit to $1$ when $\langle \alpha^{p}_e, \zeta \rangle  < 0$ and $\hbar / (1 - \hbar)$ when $\langle \alpha^{p}_e, \zeta \rangle  > 0$. The factors of type \ref{eq:thirdfactor} are similar, replacing $\langle \alpha^{p}_e, \zeta \rangle $ by $\langle \alpha^{p^!}_e, -\eta \rangle$ and $\hbar$ by $\hbar^{-1}$. In particular, we see that all factors are bounded as $t \to \infty$. 

There are $\rk \ind_p$ factors of type \ref{eq:secondfactor} with $\langle \alpha^{p}_e, \zeta \rangle  > 0$ and $\rk \ind_{p^!}$ factors of type \ref{eq:thirdfactor} with $\langle \alpha^{p^!}_e, -\eta \rangle  > 0$. The lemma follows.
\end{proof}

\section{Elliptic cohomology over the Tate curve}
Fix a coordinate $q$ on the formal punctured disk $\D^*$, and let $E = \C^* / q^{\Z}$ be the corresponding family of elliptic curves over $\D^*$. More generally, let $A$ be a complex torus with cocharacter lattice $\flavorlie_\Z$, and let $\elliptic_A = A / q^{\flavorlie_\Z}$ be the corresponding abelian variety over $\D^*$. $A$-equivariant elliptic cohomology, in the narrow sense needed here, is a covariant functor from $A$-schemes to schemes
\[ \ellcoh_A( - ) : A-\operatorname{Sch} \to \operatorname{Sch} \]
such that $\ellcoh_A(pt) = \elliptic_A$. The analogue of a class $\gamma$ in equivariant K-theory will be a section $f$ of a coherent sheaf $\genkel$ over $\ellcoh_A(\XX)$. 

To an equivariant line bundle $u \in \pic_A(\XX)$, one can associate a bundle $\Theta(u)$ over $\ellcoh_A(\XX)$ called the Thom class of $u$, with a canonical section $\vartheta(u)$.  The next few subsections explain how to do this in our setting. 

		\subsection{Line bundles on abelian varieties}

	Given an elliptic curve $E = \C^\times / q^\Z$, we can specify a line bundle on $E$ starting from the trivial bundle on $\C^\times$, by glueing the fiber over $x$ to the fiber over $qx$ by multiplying by the `factor of automorphy' $c x^d$ for some constant $c$ and integer $d$. 
	
	A holomorphic section of this line bundle may be identified with a holomorphic function $f(x)$ on $\C^\times$ such that $f(qx) = cx^d f(x)$. 
	
	We start with a line bundle $\cL$ with factor of automorphy $-q^{-1/2}x^{-1}$, which serves as a building block for most other bundles arising in the theory of elliptic stable envelopes. The theta function
	
	\begin{equation} \label{eq:thetaformula} \vartheta(x) := (x^{1/2} - x^{-1/2})\prod_{n > 0} (1 - q^nx)(1 - q^{n}/x), \end{equation}
	defined on the double cover of $\C^\times$, has precisely this automorphy and thus defines a section of $\cL$. 	In this paper, $q$ will be a formal variable, and we may think of the right-hand side of Equation \ref{eq:thetaformula} as an element of $\C[x^{\pm 1/2}][[q]]$.
	 
	Given a map of tori $u : \diag \to \C^\times$, we may define a line bundle $\Theta(u)$ on $\elliptic_\diag$ by pulling back $\cL$ via the induced map $\elliptic_\diag \to E$. $\Theta(u)$ comes with a canonical section $\vartheta(u)$ also obtained by pullback. More generally, given a virtual representation $R = \sum_\mu c_\mu t^\mu$ of $\diag$, we have the line bundle $$\Theta(R) = \otimes_{\mu} \Theta(t^\mu)^{c_\mu}$$ and (meromorphic) section $\vartheta(R)$ defined by $\prod_{\mu} \vartheta(t^{\mu})^{c_\mu}$.
	
	\subsection{Uniformization}

By expanding the expression in Equation \ref{eq:thetaformula}, the section $\vartheta(R)$ may be viewed as an element of $H^0(\diag', \mathscr{O}')[[q]]$, the completion of $H^0(\diag', \mathscr{O})[q]$ at $q=0$, where $\diag'$ is a certain finite cover of $\diag$ defined by taking the square roots of the coordinates $x$. We indicate the latter interpretation by the superscript $u$ for `uniformization', so that
	\[ \vartheta(R) \in H^0(\elliptic_\diag, \Theta(R)), \ \ \ \vartheta^u(R) \in H^0(\diag', \mathscr{O})[[q]]. \]

	\subsection{Line bundles on the scheme of elliptic cohomology}
	The ring of virtual representations of a torus $\diag$ is otherwise known as $K_\diag(pt)$. The definition of $\Theta$ can in fact be extended to $K_{\flavor}(X)$ for a torus $\flavor$ acting on a space $X$, and defines a group map $$\Theta: (K_{\flavor}(X), +) \to (\pic(\ellcoh_{\flavor}(X)), \otimes).$$ Given $R \in K_{\flavor}(X)$, we write $\vartheta(R)$ for the canonical meromorphic section of $H^0(\ellcoh_{\flavor}(X), \Theta(R))$. 
	
In the hypertoric setting, this is not much of a generalisation. The elliptic cohomology of a hypertoric variety admits a natural embedding
\[ \ellcoh_{\flavor \times \C^\times_\hbar}(X) \to \elliptic_{\diag \times \C^\times_\hbar}, \]
induced by the embedding $X \to [\mu_{\gaugeG}^{-1}(0) / \gaugeG]$. It is the elliptic analogue of the embedding $\Spec K_{\flavor \times \C^\times_\hbar}(X) \to \Spec K_{\diag \times \C^\times_\hbar}(pt)$ induced by the Kirwan map. All of our line bundles will in fact be pulled back along this map.

\subsection{Uniformization on $\ellcoh_{\flavor}(X)$}

	Using the maps $\Theta$ and $\vartheta$, we have a large supply of line bundles on $\ellcoh_{\flavor}(X)$, each equipped with a canonical section. We would like to think of these sections as elements of $K_{\flavor}(X)[[q]]$, the completion of $K_{\flavor}(X)[q]$ at $q=0$.
	
	We thus define $\vartheta^u : K_{\flavor}(X) \to K_{\flavor}(X)[[q]]$ as the dotted line in the following commutative diagram.

\begin{equation}
\begin{tikzcd}
K_{\flavor}(X) \arrow[d, dotted, "\vartheta^u"] & \arrow[l] K_\diag(pt) \arrow[d, "\vartheta^u"] \\
K_{\flavor}(X)[[q]] & \arrow[l] K_\diag(pt)[[q]]
\end{tikzcd}	
\end{equation}
Here the top horizontal map is the Kirwan map, and the bottom horizontal is induced by the Kirwan map. 

To see that the result does not depend on the choice of lift to $K_{\mathbf{D}}(pt)$, we note that an element of $K_{\mathbf{A}}(X)$ is determined by its restrictions to fixed points $p \in X^{\mathbf{A}}$. It thus suffices to check independence of lift for the surjection $K_{\mathbf{A}}(pt) \leftarrow K_{\mathbf{D}}(pt)$. The kernel of this map is spanned by elements $t^{\mu} - t^{\mu'}$, where $\mu, \mu'$ are characters of $\mathbf{D}$ which coincide after restriction to $\mathbf{A}$. Independence of lift follows from the fact that such elements map to $1 \in K_{\mathbf{A}}(pt)[[q]]$.  

	\section{Elliptic stable envelopes}
	
	We fix the following data:  
	\begin{enumerate}
		\item A sufficiently generic cocharacter $\sigma$ of $\finsymp$.
		\item A polarization $T^{1/2}_X$. 
	\end{enumerate}
	
Let $p \in X^\flavor$. Aganagic and Okounkov \cite{aganagic2016elliptic} associate to this data an elliptic stable envelope, which is a section of a certain line bundle $\Theta(R)$ on a certain enlargement of the elliptic cohomology scheme. 

For a complete definition (which also applies to situations with non-isolated fixed locus), we refer the reader to \cite{aganagic2016elliptic}. We will in fact work with the `renormalized' elliptic stable envelopes, which are described in Smirnov and Zhou \cite{Smzhou20}. They are given by a simple formula in terms of the so-called duality interface, defined as follows. 

Let $S = \sum_{e \in E} \chi_e \check{\chi}_e$. We have a bundle $\Theta(S)$ on  $\elliptic_\diag \times \elliptic_{\diag^{\vee}}  \times \elliptic_{\C^\times_\hbar}$ with a canonical section $\vartheta(S)$.

 Let $\C^\times_\hbar$ act antidiagonally on $X \times X^!$, so that the character denoted $\hbar$ in $K_{\C^\times_\hbar}(X)$ (resp. $K_{\C^\times_\hbar}(X^!)$) pulls back to $\hbar$ (resp. $\hbar^{-1}$). We can pull back $ \Theta(S)$ along the embedding 
	\[ \ellcoh_{\flavor \times \gaugeG^{\vee} \times \C^\times_\hbar}(X \times X^!) \to \elliptic_\diag \times \elliptic_{\diag^{\vee}} \times \elliptic_{\C^\times_\hbar} \]
	to obtain a line bundle on the left-hand side, which we denote $\frak{M}$. We write $\frak{m}$ for the restriction of $\vartheta(S)$. The following will serve for us as a definition.
	\begin{theorem} \cite{okounkovstringmath2019talk, smirnov20203d}
		The renormalized elliptic stable envelope of $p$ on $\XX$ is the restriction of $\frak{m}$ to $X \times p^!$.
\end{theorem}

\section{The class $\xi$ for loop spaces and the duality interface}
\subsection{Loop spaces}
We recall some concepts and notation from \cite{MasheYau20}. Starting from the data defining a hypertoric variety, namely a set $E$, a subtorus $\gaugeG \to (\C^\times)^E$ and a character $\eta$ of $\gaugeG$, that paper defined a loop analogue of $X$ denoted $\tloops X$. It is, loosely speaking, the infinite dimensional hypertoric variety associated to the data
\begin{itemize}
	\item $\loops E := E \times \Z$.
	\item $1 \to \gaugeG \to \loops \diag = (\C^\times)^{\loops E} \to \mathscr{T} \to 1$.
	\item The character $\eta$ of $\gaugeG$.
	\end{itemize}
Morally, $\tloops X$ is the symplectic reduction $\cotan  (\C^{\loops E}) \sslash_\eta \gaugeG$. In fact, the space $\tloops X$ is constructed as a limit of closed embeddings of finite dimensional hypertoric varieties $...\to \tloops_{N-1} X \to \tloops_N X \to \tloops_{N+1} X \to...$, associated to the following `truncated' hypertoric data.
\begin{itemize}
	\item $\loops_N E := E \times [-N,N]$.
	\item $1 \to \gaugeG \to \loops_N \diag = (\C^\times)^{\loops_N E} \to \mathscr{T}_N \to 1$. 
	\item The character $\eta$ of $\gaugeG$.
	\end{itemize}
The natural coordinates on $\cotan  \C^{\loops E} $ are denoted $x_{e,k}, y_{e,k}$, and correspond to the fourier modes in the expansion of a loop $(x_e(t), y_e(t)) = (\sum_{k \in \Z} x_{e,k}t^k, \sum_{k \in \Z} y_{e,k}t^k)$. The coordinate $x_{e,k}$ is `paired', under the symplectic form, with the coordinate $y_{e, -k}$. 

$\tloops X$ carries an action of an infinite-dimensional torus of Hamiltonian transformations, containing the subtorus $\flavor \times \C^\times_q$ corresponding to the action of $\flavor$ on $X$ and the action of $\C^\times_q$ by `loop rotation'. 

Given any character of $\diag \times \C^{\times}_q \times \C^\times_\hbar$, we obtain by descent a $\flavor \times \C^\times_q \times \C^{\times}_\hbar$-equivariant line bundle on $\tloops X$. We denote the bundle associated to the character of $x_{e,0}$ by $u_e$. The bundle associated to $x_{e,k}$ is $q^k u_e$ where $q$ is the tautological character of $\C^\times_q$.

\subsection{K-theory of the loop space}

We will have need for the following construction. Let $\frak{K}$ be an algebra over the ring $\C[q,q^{-1}]$ of Laurent polynomials. Choose a $\C[q]$-lattice $\frak{K}_0 \subset \frak{K}$. Let $\frak{K}_0^{\wedge}$ be the completion of $\frak{K}_0$ at the ideal $(q) \subset \C[q]$, and let $\frak{K}^{\wedge} := \frak{K}_0^{\wedge} \otimes_{\C[[q]]} \C((q))$. It is an algebra over $\C((q))$. 

 In our setting, $\frak{K}$ will be an equivariant K-group, $q$ will be a generator of $K_{\C^\times_q}(pt)$, and we may choose the lattice to be defined by those elements whose restriction to the $\flavor$-fixed locus are polynomial in $q$. 

Let $ \tloops_N K(X) := K_{\flavor \times \C^\times_\hbar \times \C^\times_q}(\tloops_N X)$. It is a $K_{\C^\times_q}(pt) = \C[q,q^{-1}]$-algebra. We define
\[  \tloops K(X) :=  \left( \varprojlim_{N \to \infty}  \tloops_N(X) \right)^{\wedge} \]

If we take $\C^\times_q$ to act trivially on $X$, then the embedding $\cst_N : X \to \tloops_N X$ of the `constant loops' is $\C^\times_q$-equivariant, and we get a restriction map 
$\cst_N^* :  \tloops_N K(X) \to K_{\flavor \times \C^\times_\hbar}(X)[q,q^{-1}]. $ In the limit, we obtain a map  $$\cst^* :  \tloops K(X) \to K_{\flavor \times \C^\times_\hbar}(X)((q)). $$ 
The class $q^k u_e$ pulls back to the class of the same name on the right-hand side. 
\subsection{Periodisation}
 
Recall the diagonal embedding $\flavor \to \loops_N \diag / \gaugeG = \mathscr{T}_N$, and dually $\mathscr{T}_N^{\vee} \to \dualflavor $. Given a character $\zeta$ of $\dualflavor $, let $\hat{\zeta}$ be its pullback along this last map. 

Fix a generic character $\zeta$ of $\dualflavor $, and consider the hypertoric data
\begin{itemize}
	\item $\loops_N E := E \times [-N,N]$.
	\item $\mathscr{T}^{\vee}_N \to \loops_N \diag^{\vee} \to \gaugeG^{\vee}$. 
	\item The character $\hat{\zeta}$ of $\mathscr{T}^{\vee}_N$.
	\end{itemize}
	
 The associated hypertoric space $\per_N X_{\hat{\zeta}}^!$ carries an action of $\gaugeG^{\vee} = \loops \diag^{\vee} / \mathscr{T}^{\vee}$.

In contrast to the situation considered in \cite{MasheYau20}, our character ${\hat{\zeta}}$ is `trivial in the loop direction'. In particular, the periodic hyperplane arrangement associated to $\lim_{N \to \infty} \per_N X^!_{\hat{\zeta}}$ has period $0$, and collapses to the finite arrangement associated to $X^!$. 

Let $\loops_N \dualflavor  := (\dualflavor )^{[-N,N]}$, $\loops_N \gaugeG := \gaugeG^{[-N,N]}$. We have 
\begin{equation} \label{eq:subtorusofscrT} 1 \to \loops_N \dualflavor  \to \mathscr{T}^{\vee}_N \to \mathscr{\gaugeG}_N^{\vee} \to 1 \end{equation}
where $\mathscr{\gaugeG}_N := \loops_N \gaugeG / \gaugeG$.

We have $$\per_N X_{\hat{\zeta}}^! = T^*  \C^{\loops_N E} \sslash_{\widetilde{\zeta}} \mathscr{T}^{\vee}.$$
Using \ref{eq:subtorusofscrT}, we can factor this as $$\left(\cotan  \C^{\loops_N E} \sslash_{\widetilde{\zeta}} \loops_N \dualflavor  \right) \sslash \mathscr{\gaugeG}_N^{\vee} = (X^!_\zeta)^{[-N,N]} \sslash \mathscr{\gaugeG}_N^{\vee}.$$
This last reduction is by definition the scheme-theoretic $\mathscr{\gaugeG}_N^{\vee}$-quotient of the moment fiber $$\mu^{-1}_{\mathscr{\gaugeG}_N^{\vee}}(0) = X^!_\zeta \times_{\frak{g}} X^!_{\zeta} \times_\frak{g} ... \times_\frak{g} X^!_\zeta.$$
Here the right-hand side denotes the fiber product of $2N+1$ copies of $X^!_\zeta$ over the $\gaugeG^{\vee}$-moment map $X^!_\zeta \to \frak{g}$. The action of $ \mathscr{\gaugeG}_N^{\vee}$ on this scheme has positive-dimensional stabilisers, and it will be useful to also consider the stack quotient $[\mu^{-1}_{\mathscr{\gaugeG}_N^{\vee}}(0) / \mathscr{\gaugeG}_N^{\vee}]$.

\begin{remark} The scheme-theoretic quotient can be described as follows. For a positive integer $K$, let $\gaugeG^{\vee}[K-\operatorname{tor}] \subset \gaugeG^{\vee}$ be the subtorus of $K$-torsion points. The diagonal embedding 
\begin{equation} \label{eq:diagonalembedding} \Delta_N : X^!_\zeta \to \mu^{-1}_{\mathscr{\gaugeG}_N^{\vee}}(0) \end{equation} descends to an isomorphism of schemes
\[ X_\zeta^! / \gaugeG^{\vee}[2N+1 - \operatorname{tor}] \cong \per_{N} X_{\hat{\zeta}}^!. \]
\end{remark}

\subsection{K-theory of the periodisation}


We write $K^{\circ}(X)$ for the K-group of equivariant perfect complexes. Consider 
\[ \per_N  K^{\circ}(X) := K^{\circ}_{\C^\times_\hbar \times \gaugeG^{\vee}}\left([\mu^{-1}_{\mathscr{\gaugeG}_{N}^{\vee}}(0) / \mathscr{\gaugeG}_N^{\vee}] \right) = K^{\circ}_{\C^\times_\hbar \times \loops_N \gaugeG^{\vee}}\left(\mu^{-1}_{\mathscr{\gaugeG}_{N}^{\vee}}(0) \right).\] 
Projection defines a $\C^\times_\hbar \times \loops_N \gaugeG^{\vee}$-equivariant map

$$\mu^{-1}_{\mathscr{\gaugeG}_{N+1}^{\vee}}(0) \to \mu^{-1}_{\mathscr{\gaugeG}_{N}^{\vee}}(0).$$ 

Pullback along these maps defines a direct system over $\Z^{\geq 0}$, and we define the direct limit
\[  \per K^{\circ}(X^!) := \varprojlim_{N \to \infty}  \per_N K^{\circ}(X^!) \]
This ring is intended as a substitute for the K-theory of (a stacky version of) the limit space $\per X^!_{\hat{\zeta}}$, which we do not attempt to define here. 
 
Pullback along \ref{eq:diagonalembedding} defines maps 
\[ \Delta_N^* : \per_N K^{\circ}(X^!) \to K^{\circ}_{\gaugeG^{\vee} \times \C^\times_\hbar}(X^!_{\zeta}) =  K_{\gaugeG^{\vee} \times \C^\times_\hbar}(X^!_{\zeta})\] 
which combine to define a map \[ \Delta^* : \per K^{\circ}(X^!) \to K_{\gaugeG^{\vee} \times \C^\times_\hbar}(X^!_{\zeta}). \]
Let $\check{u}_{e,k} \in \per K^{\circ}(X^!)$ be the class associated to the character $\check{\chi}_{e,k}$ of $\loops \diag^{\vee}$. Then by construction, we have

\[  \Delta^* \check{u}_{e,k} = \check{u}_e. \]

\subsection{Combining $\tloops X$ and $\per X^!$} 
\begin{definition} 
We consider the `completed' tensor product 
\[ \tloops K(X) \hat{\otimes} \per K(X^!) := \varprojlim_{N \to \infty} \left( \tloops_N K(X) \otimes_{\C[\hbar, \hbar^{-1}]} \per K^{\circ}(X^!) \right)^{\wedge}, \]
where the $\C[\hbar]$-algebra structure on the right-hand factor is twisted by $\hbar \to \hbar^{-1}$, and the $\wedge$ superscript denotes completion at $q=0$ as before.
\end{definition}

Taking the limit of the maps $(\cst_N \times \Delta_N)^*$, we obtain a map
\begin{equation} \label{eq:phi} (\cst \times \Delta)^* :  \tloops K(X) \hat{\otimes} \per K(X^!) \to K_{\flavor \times  \C^\times_\hbar}(X) \otimes_{\C[\hbar, \hbar^{-1}]} K_{\gaugeG^{\vee} \times \C^\times_\hbar}(X^!) ((q)). \end{equation}
     
 \subsection{Polarization by positive loops}
We fix the polarisation of $\tloops X$ by the positive loops, meaning the polarization induced by the lagrangian subspace $\loops^+(\cotan \C^E) \subset \cotan  (\loops \C^{E}) $ defined by $$\{ x_{e,k} = 0 | k < 0 \} \cap \{ y_{e,k} = 0 | k \leq 0 \}.$$
\begin{remark}
Another natural polarization of $\tloops X$ is induced by the subspace $\loops \C^{E}  \subset \cotan  (\loops \C^{E}) $, corresponding to loops in the $x$-variables. We have no use for it here.
\end{remark}

As a representation of $\loops \diag \times \C^\times_\hbar$, the space $\loops^+ (\cotan \C^E)$ decomposes as a sum of characters
\[ \loops^+ (\cotan \C^E) = \sum_{e \in E} \left( \sum_{k \geq 0} \chi_{e, k} + \sum_{k < 0} \hbar^{-1} \chi^{-1}_{e, k} \right). \]

We write $\check{\chi}_{e,k}$ for the dual characters of $\loops \check{\diag}$, appearing in the symplectically dual loop space. We write $ \loops^+_N (\cotan \C^E)$ for the analogous character of $\loops_N \diag \times \C^\times_\hbar$.

\subsection{The universal intertwiner and the duality interface}
Let $\widetilde{\xi}(\loops^+_N (\cotan \C^{E}))$ be class defined as in Section \ref{sec:xi} and Remark \ref{rem:xi} starting from the subspace $\loops^+_N (\cotan \C^{E}) \subset \cotan  \loops_N \C^E$. It is given by the formula
\begin{equation} \label{eq:thepreimage}  \prod_{e \in E} (1 - \chi_{e,0}\check{\chi}_{e,0}) \prod_{0 < k \leq N} ( 1 - \chi_{e,k} \check{\chi}_{e,k})\prod_{- N \leq k < 0} (1 - \chi^{-1}_{e,k}\check{\chi}^{-1}_{e,k}) \end{equation}
 Taking the image of \ref{eq:thepreimage} under the Kirwan map, we obtain an element $$\xi(\loops^+_N (\cotan \C^{E})) \in \tloops_N K(X) \otimes_{\C[\hbar, \hbar^{-1}]} \per_N K(X^!).$$
Letting $N \to \infty$, these classes define an element 
\[ \xi(\loops^+ (\cotan \C^{E})) \in \tloops K(X) \hat{\otimes} \per K(X^!). \]

\begin{definition}
Let $\xi(\loops^+)$ denote the image of $\xi(\loops^+ (\cotan \C^{E}))$ under the map \ref{eq:phi}.
\end{definition} 
Thus $\xi(\loops^+)$ is a class in $K_{\flavor \times  \C^\times_\hbar}(X) \otimes_{\C[\hbar, \hbar^{-1}]} K_{\gaugeG^{\vee} \times \C^\times_\hbar}(X^!) ((q))$. We can write it in terms of the tautological classes as
\begin{equation} \label{eq:xionloops} \xi(\loops^+) = \prod_{e \in E} (1 - u_e \check{u}_e)\prod_{k > 0} ( 1 - q^k u_{e} \check{u}_{e}) (1 - q^k u^{-1}_{e}\check{u}^{-1}_{e}).  \end{equation}

Comparing Equation \ref{eq:xionloops} with the definition of the duality interface $\frak{m}$, we obtain our final result:
\begin{theorem} \label{thm:maintheorem}
	The class $\xi(\loops^+)$ equals the uniformization $\frak{m}^u$ of the duality interface $\frak{m}$, multiplied by the fractional bundle $\prod_e (u_e \check{u}_e)^{-1/2}$. 
\end{theorem}

	\bibliographystyle{amsalpha}
	\bibliography{symplectic}
	
		\noindent{[1] Chinese University of Hong Kong, Room 235, Lady Shaw Building,
		Shatin, N.T., Hong Kong, China}\\\\
	\noindent{[2] Massachusetts Institute of Technology (MIT), IAiFi Institute, 182 Memorial Drive, Cambridge, MA 02139, USA}\\\\
	\noindent{[3] National Research University Higher School of Economics, Russian Federation, Lab- oratory of Mirror Symmetry, NRU HSE, 6 Usacheva str., Moscow, Russia}\\\\
	\noindent{[4] Yanqi Lake Beijing Institute for Mathematical Sciences and Applications (BIMSA). Huairou, Beijing, China } \\\\
	\noindent{[5] {Yau center of mathematics, Tsinghua university,  Beijing, China }\\\\
	\noindent{\tt{mcb@math.cuhk.edu.hk, artan@mit.edu, styau@tsinghua.edu.cnu}}

\end{document}